\documentclass[11pt,reqno]{amsart}
\usepackage{graphicx}
\pdfoutput=1
\usepackage{amsfonts,amsmath,amssymb}
\usepackage{hyperref}
\usepackage{paralist}
\usepackage[linesnumbered,ruled,vlined,norelsize]{algorithm2e}

\usepackage{thmtools}
\declaretheoremstyle[bodyfont=\normalfont]{noncursive}
\declaretheorem{theorem}

\declaretheorem[numberlike=theorem]{proposition}
\declaretheorem[numberlike=theorem]{corollary}
\declaretheorem[style=noncursive,numberlike=theorem]{definition}

\declaretheorem[style=noncursive,numberlike=theorem]{remark}

\declaretheorem[style=noncursive,numberlike=theorem]{problem}
\newcommand{\CC}[1]{\mathbb{C}^{#1}}

\newcommand{\RR}[1]{\mathbb{R}^{#1}}

\newcommand{\lr}{\longrightarrow}

\numberwithin{equation}{section}

\title[Normal form for second order differential equations]{Normal form for second order differential equations}

	\author {Ilya Kossovskiy }
	\address{\parbox{0.8\linewidth}{%
	               Faculty of Mathematics, University of Vienna, Austria/\\
	               Department of Mathematics and Statistics, Masaryk University, Brno, Czechia}
	    }
	\email{kossovi3@univie.ac.at, kossovskiyi@math.muni.cz}
	
	\author {D. Zaitsev*}
\address{School of Mathematics, Trinity College, Dublin}
\email{zaitsev@maths.tcd.ie}

\begin{document}

\begin{abstract}
Applying methods of {\em CR-geometry}, we solve the local equivalence problem for second order (smooth or analytic) ordinary differential equations. We do so by presenting a {\em complete convergent normal form} for this class of ODEs. The normal form 
is optimal in the sense that it is defined up to the automorphism group of the model (flat) ODE $y''=0$. For a generic ODE, we also provide a unique normal form. By doing so, we give a solution to a problem which remained unsolved since the work of Arnold \cite{arnoldgeom}. The method can be immediately applied to important classes of second order ODEs, in particular, the Painlev\'e equations.

As another application of the convergent normal form, we discover distinguished curves  associated with a differential equation that we call {\em chains}. 
\end{abstract}

\maketitle

\date{\today}

\tableofcontents

\section{Introduction}

\subsection{Overview of the normal form problem} The aim of the paper is to use methods of {\em Cauchy-Riemann geometry} (shortly CR-geometry) for solving the following well known problem: 
construct a {\em complete} normal form for 
a (real or complex) second order ordinary differential equation (ODE)
\begin{equation}\label{ODE}
\mathcal E=\bigl\{y\rq{}\rq{}=F(x,y,y\rq{})\bigr\},
\end{equation}
considered locally near a fixed reference point $p=(x_0,y_0, u_0)$,
representing a fixed solution $y=\varphi(x)$ with $u_0 = \varphi'(x_0)$.
Without loss of generality, we assume $(x_0,y_0,u_0)=0$. 
The function $F$ in \eqref{ODE} is assumed to be either $C^\infty$ near the origin in $\RR{3}$, or holomorphic near the origin in $\CC{3}$, depending on the context. In what follows, we do our considerations in the real case only, as in the complex case they are very analogues. By using  the bundle $ J^1(\RR{},\RR{})\simeq\RR{3}(x,y,u)$ of $1$-jets of functions $\RR{}\mapsto\RR{}$ (where the jet coordinate $u$ corresponds to $\frac{dy}{dx}$) as well as the bundle $ J^2(\RR{},\RR{})\simeq\RR{4}(x,y,u,v)$ of $2$-jets of functions $\RR{}\mapsto\RR{}$, we may interpret the ODE $\mathcal E$  as a hypersurface in $J^2(\RR{},\RR{})$.
For ODEs \eqref{ODE} considered locally near the origin, we study the problem of their classification  under  {\em point transformations}, i.e. local self-transformations of $\RR{2}$, sending families of solutions of two ODEs into each other, and moreover,  sending into each other the given distinguished solutions.

The natural classification problem under discussion was first considered in the celebrated paper of Tresse \cite{tresse} who found a complete system of {\em relative differential invariants} for the problem (a relative invariant of order $k$ is a function $I$ on the space $j^k(\RR{3},\RR{})$ of $k$-jets of defining functions $F(x,y,u)$, as in \eqref{ODE}, such that for any point diffeomorphism $H$, the transformed equation 
$y'' = \tilde F(x,y, y')$  
satisfies an identity 
$$
	I(j^k \tilde F)=\lambda\cdot I(j^k F)
$$ 
for an appropriate {\em non-vanishing} factor $\lambda(x,y,u)$).  In particular, the lowest possible order of a relative differential invariant equals to $4$, and the identical vanishing of the (only) two order $4$ invariants  
\begin{equation}\label{inv}
\begin{aligned}
I_1(F):=&F_{u^4}, \\ 
I_2(F):=&D^2\bigl(F_{u^2}\bigr)-F_{u}D\bigl(F_{u^2}\bigr)-4D\bigl(F_{yu}\bigr)+6F_{y^2}-3F_yF_{u^2}+4F_{u}F_{yu}
\end{aligned}
\end{equation}
is necessary and sufficient for $\mathcal E$ to be locally equivalent to the  {\em flat ODE}
\begin{equation}\label{flat}
y\rq{}\rq{}=0.
\end{equation}
Here $D$ is the total differentiation operator 
\begin{equation}\label{DD}
D:=\frac{\partial}{\partial x}+u\frac{\partial}{\partial y}+F\frac{\partial}{\partial u}.
\end{equation}
  In this way, Tresse obtained a complete and effective solution for the problem of ``flattening'' an ODE \eqref{ODE}. It should be mentioned that, in fact, the flattening problem here is equivalent to the {\em linearization problem}, as every linear ODE \eqref{ODE} is always transformable into the flat ODE \eqref{flat} in a neighborhood of any point where it is defined (see e.g.\ \cite{arnoldgeom}). Importantly, Tresse's relative invariants allow for producing {\em absolute} invariants for the equivalence problem, and subsequently solving the equivalence problem itself. We refer to Kruglikov \cite{kruglikov} for more details here, as well as for a modern treatment of Tresse's theory.

The classification problem in its full generality was further attacked in the work of E.\,Cartan \cite{cartanODE}, who considered the problem in the framework of the equivalence problem for $G$-structures. In this work, Cartan successfully applied his 
method of reducing to the equivalence problem for $\{e\}$-structures. 
In particular, he re-discovered the two Tresse's fundamental invariants \eqref{inv}.  Cartan's approach allows for deciding on the equivalence of {\em any} two given ODEs \eqref{ODE} {\em in principle}. 
However in practice, applying Cartan's method requires solving equations
defining Cartan's connection and subsequently, 
solving the equivalence problem for the $\{e\}$-structures
on $8$-dimensional projective bundles with fiber dimension $5$.
The latter steps may present significant technical difficulties. An example illuminating the difficulties here is the fact that even the particular class of {\em Painlev\'e differential equations} is not classified yet under point transformations. For various results on classifying classes of second order ODEs (including Painlev\'e equations), we refer to the work of  Kamran, Lamb and Sharwick \cite{kamran}, Babich and Bordag \cite{bb}, Bandle and Bordag \cite{bb1}, Dmitrieva and Sharipov \cite{ds},  Bagredina \cite{bagredina}, Kartak \cite{kartak}, Yumaguzhin \cite{yumagu}.

In view of the above, an effective  solution of the equivalence problem for ODEs \eqref{ODE} remains to be of great interest. A natural approach here is presenting a {\em normal form} for an ODE \eqref{ODE}. An important step 
in this direction is due to Arnold \cite{arnoldgeom}, 
who proved that, via a series of geometrically motivated transformations,
any ODE  \eqref{ODE} 
can be locally transformed into the (pre-)normal form 
\begin{equation}\label{arnolds}
y''=N(x,y,y')=A(x)y^2+O(y^3)+O(y'^3)
\end{equation}
for an appropriate (semi-)invariant function $A(x)$. (Here for a given monomial $x^ky^ly'^m$, we denote by $O(x^ky^ly'^m)$  a respectively smooth or anaytic function of the kind $x^ky^ly'^m\cdot \lambda(x,y,y')$). We note, however, that Arnold's normal form 
\eqref{arnolds}
is {\em incomplete} in the sense that the same ODE can have an infinite-dimensional space of normal forms, as illustrated by the following example:

\medskip

\noindent {\bf Example.} Let us consider the class of ODEs
\begin{equation}\label{lin}
y''= xV(y)(y')^3,
\end{equation}
where $V(y)$ is any (smooth or analytic) function defined near the origin.
Then \eqref{lin} is in Arnold's (pre-)normal form \eqref{arnolds}. However, {\em any} ODE \eqref{lin} is diffeomorphic to the flat ODE \eqref{ODE} near any solution. 
Indeed,
in a neighborhood of any solution with $y'\neq 0$, we can switch
$x$ and $y$
and a straightforward calculation shows that then \eqref{lin} becomes
$$x''=-V(y)x,$$
i.e. our ODE is {\em linear} in the new (exchanged) coordinates. 
This proves that \eqref{lin} can be linearized by the diffeomorphism 
$(x,y)\mapsto (y,x)$, and hence \eqref{lin} can be further flattened. The fact that it can be flattened in a neighborhood of a solution with  vanishing $y'$ follows by uniqueness from Tresse's theorem on the basic invariants \eqref{inv}. We conclude that the flat ODE \eqref{ODE} has an infinite dimensional space \eqref{lin} of Arnold's (pre-)normal forms. 

\medskip

In this way, Arnold's normal form does {\em not} solve the equivalence problem, and the arising problem of completely classifying ODEs \eqref{ODE} 
(up to finitely many parameters) via a normal form still remains open.  
It is the main goal of this paper is to provide such a complete normal form,
thereby providing a solution to the classification problem.

\subsection{Complete normal form for a second order ODE}  To formulate our result in detail, we consider  the canonical $1$-form
$$\omega:=dy-udx$$
 in the $1$-jet space $J^1(\RR{},\RR{})$. It determines the canonical $2$-distribution 
\begin{equation}\label{PI}
\Pi=\{\omega=0\}.
\end{equation}
Now our main can be formulated is as follows.  

\begin{theorem}\label{theor1}
Let 
\begin{equation}\label{ode}
	y'' = F(x,y,y')
\end{equation}
 be a real $C^\infty$ smooth (resp.\ holomorphic) second order ODE, defined in a neighborhood $U$ of an element  
 $J_0:=(x_0,y_0,u_0)\in J^1(\RR{},\RR{})$ 
 (resp.\ $J_0:=(x_0,y_0,u_0)\in J^1(\CC{},\CC{})$). 
 Let  $\Pi_0\subset J^1(\RR{},\RR{})$ be the value of the distribution \eqref{PI} at $(x_0,y_0,u_0)$. Then there exists local $C^\infty$ smooth (resp.\ holomorphic)
coordinates in $\RR{2}$ (resp.\ $\CC{2}$),
defined in a neighborhood of $(x_0,y_0)$, transforming the element 
 $J_0$ into the origin in the corresponding jet space, 
 and the ODE \eqref{ode} into the normal form
 \begin{equation}\label{normalODE}
	y\rq{}\rq{}=N(x,y,y'), 
	\quad
	N=O\bigl(x^2y\rq{}^2\bigr)+O\bigl(y\rq{}^4\bigr).
\end{equation}
A transformation  $H=(f,g)$ bringing 
\eqref{ode}
into a normal form 
is uniquely determined by 
the restriction at $J_0$ of the differential of the prolonged map $H^{(1)}$ onto $\Pi_0$, and the transverse second order derivative of  
$H^{(1)}$ at $J_0$.
In turn, in any coordinates where $(x_0,y_0,u_0)$ is the origin, this amounts to a choice of the five parameters
\begin{equation}\label{mapdata}
f_x(0,0)=s,\quad   g_y(0,0)=t,\quad f_y(0,0)=\alpha,\quad f_{xx}(0,0)=\beta,\quad g_{yy}(0,0)=r,\quad s,t,\neq 0.
\end{equation}

If the ODE \eqref{ode} is real-analytic, then the normal form \eqref{normalODE} as well as the normalizing transformation are real-analytic.

\end{theorem}

\begin{remark}\label{approx}
In case the ODE $\mathcal E$ in \autoref{theor1} is respectively $\RR{}$- or $\CC{}$-analytic, the normal form \eqref{normalODE} reads as
\begin{equation}\label{normalNkl}
y''=\sum_{k,l\geq 0} F_{kl}(y)x^k(y')^l, \quad \mbox{where}\quad F_{k0}=F_{k1}=0,\,k\geq 0, \quad F_{02}=F_{03}=F_{12}=F_{13}=0,
\end{equation}
and the functions $F_{kl}$ are respectively $\RR{}$- or $\CC{}$-analytic. 
\end{remark}

As an application of the (convergent) normal form, we discover the following canonical object associated with an ODE \eqref{ODE}.

\begin{theorem}\label{Thmchains}
In the setting of \autoref{theor1}, there exists a distinguished smooth (resp. holomorphic) invariant $2$-parameter family of pairwise transverse curves in $U$ called chains, all of which are transverse to the distinguished $2$-plane $\Pi_0$. In any direction $d$ in $J^1(\RR{},\RR{})$ 
 (resp. $J^1(\CC{},\CC{})$) transverse to $\Pi_0$, there exists a unique chain passing through $(x_0,y_0,u_0)$. Each chain through $(x_0,y_0,u_0)$ takes the form
$$\bigl\{x=u=0\bigr\}$$
in some normal form 
at the point $(x_0,y_0,u_0)$. 
If the ODE \eqref{ode} is in addition real-analytic, 
then the family of chains is also real-analytic.
\end{theorem}

\begin{remark}\label{optimal}
We emphasize that the normal form in \autoref{theor1} is an {\em optimal} natural normal form for ODEs \eqref{ODE}, in the sense that a normalizing transformation is defined uniquely up to the stability group $G$ of the model ODE \eqref{flat} (i.e. the subgroup in the symmetry group preserving the distinguished point $(0,0,0)$ in $J^1(\RR{},\RR{})$). The latter group $G$ consists of the projective transformations
\begin{equation}\label{G}
x\mapsto \frac{a_1x+a_2y}{1+c_1x+c_2y},\quad y\mapsto \frac{b_1y}{1+c_1x+c_2y},\quad a_1,b_1\neq 0.
\end{equation}  
It is not difficult to set up a one-to-one correspondence between the parameters in the group \eqref{G} and the ones in \eqref{mapdata}. Thus, {\em the number of parameters determining a natural normal form for ODEs \eqref{ODE} can not be reduced further.}
\end{remark}

\begin{remark}\label{IviaNF}
It is not difficult to verify  that Tresse's fundamental invariants \eqref{inv} at a distinguished point $p$ are read from a normal form at this point as
\begin{equation}\label{findI}
I_1(F)=N_{(y')^4}(0,0,0), \quad I_2(F)=N_{x^2(y')^2}(0,0,0). 
\end{equation}
Further, it is not difficult to see that all the coefficients   $c_{jlm}$ for non-zero Taylor polynomials $c_{jlm}x^jy^ly'^m$ in the normal form \eqref{normalODE} can be expressed polynomially via the respective Tresse's semi-invariants of weight $k:=j+l+m$.
\end{remark}

\subsection{Special normal form} Next, we are able to provide a {\em unique} normal form for ODEs \eqref{ODE} satisfying a simple generic condition. Let us introduce

\begin{definition}\label{flatpoint} An ODE \eqref{ODE}  is called {\it flat} at a point  $(x_0,y_0,u_0)$ if 
$$I_1(F)|_{(x_0,y_0,u_0)}=I_2(F)|_{(x_0,y_0,u_0)}=0.$$
On the other hand, an ODE is called {\it strongly nonflat} at the point $(x_0,y_0,u_0)$, if
$$I_1(F)|_{(x_0,y_0,u_0)}\neq 0,\quad I_2(F)|_{(x_0,y_0,u_0)}\neq 0.$$
If finally, only one of the above values vanishes, an ODE is called {\em semi-flat} at $(x_0,y_0,u_0)$.
\end{definition}
As follows from the theory of Tresse, \autoref{flatpoint} is invariant under diffeomorphisms. Moreover, 
   \emph{an ODE $\mathcal E$, as in \eqref{ODE}, defined near a point $(x_0,y_0,u_0)$ is  equivalent to the flat ODE \eqref{flat} near $(x_0,y_0,u_0)$ if and only if there exists a neighborhood of $(x_0,y_0,u_0)$ in $ J^1(\RR{},\RR{})$ at each  point of which the ODE $\mathcal E$ is flat}. 
   
Clearly, by performng a scaling in $x,y$ the constants $F_{22}(0),F_{04}(0)$ in \autoref{approx} can be normalized to be equal to $1$ at a strongly nonflat point. Moreover, it turns out that at a strongly nonflat point the normal form \eqref{normalODE} can be chosen uniquely, up to a $\mathbb{Z}_n$ action.
   
 \begin{theorem}\label{snf}
 Let $\mathcal E$, as in \eqref{ODE}, be a real $C^\infty$ smooth (resp. $\CC{}$-analytic) second order ODE, defined in a neighborhood $U$ of an element  $(x_0,y_0,u_0)\in J^1(\RR{},\RR{})$ (resp. $(x_0,y_0,u_0)\in J^1(\CC{},\CC{})$). Assume that the point $(x_0,y_0,u_0)$ is strongly nonflat. Then  there exists a $C^\infty$ smooth (resp. holomorphic)  map $H$ of $(\RR{2},0)$ (resp. $(\CC{2},0)$) transforming the element $(x_0,y_0,u_0)$ to the origin and the ODE \eqref{ODE} to the special normal normal form
$$ y\rq{}\rq{}=N(x,y,y')=O\bigl(x^2y\rq{}^2\bigr)+O\bigl(y\rq{}^4\bigr),$$
where, in addition, 
\begin{equation}\label{sn}
N_{(y')^4}(0,0,0)=N_{x^2(y')^2}(0,0,0)=1,\quad N_{x^2(y')^3}(0,0,0)=N_{x(y')^4}(0,0,0)=N_{x^3(y')^3}(0,0,0)=0.
\end{equation}
In the real case, a normal form \eqref{sn} of $\mathcal E$ is unique up to the linear $\mathbb Z_2$ action
\begin{equation}\label{Z2}
x\mapsto -x, \quad y\mapsto y
\end{equation} 
on the normal forms. 
Furthermore, if the ODE $\mathcal E$ is real-analytic, then the normal form \eqref{normalODE} as well as the normalizing transformation are real-analytic.

In the complex case,  a normal form \eqref{sn} of $\mathcal E$ is unique up to the linear $\mathbb Z_8$ action
\begin{equation}\label{Z8}
x\mapsto \varepsilon x, \quad y\mapsto \varepsilon^{-2}y, \quad \varepsilon^8=1.
\end{equation} 
on the normal forms. 
 \end{theorem}  
 For example, it is easy to verify from \autoref{snf} that two real ODEs of the kind
 $$y''=y'^2(x^2+y'^2)(1+cy), \quad c>0$$
 are all point inequivalent at the origin for different values of $c$.

Next, we obtain the following
 
 \begin{corollary}
For an ODE \eqref{ODE}, given at a strongly nonflat point $p$, there is a canonical way of choosing a distinguished curve transverse to the given solution (a {\em distinguished chain}). The distinguished chain is given by the equation $\{x=u=0\}$ in any special normal form \eqref{sn} at $p$.  
\end{corollary}
For a generic ODE \eqref{sn}, it is possible to do a further normalization and get rid of the $\mathbb Z_n$ action on normal forms. We though do not provide here the respective computation and leave the details to the reader.
 
\subsection{Partial normal forms} We have now an important

\begin{remark}\label{important}
In fact, the normal form construction provided in the paper allows for the following useful application: two given ODEs can be distinguished by means of their {\em partial normal forms}, or {\em normal forms of a finite order $m,\,m \geq 5$}. (The precise integer $4$ comes from the fact that the first possible non-zero term in special the normal form \eqref{sn} has degree $5$ in $x,y,y'$). As we will see from the construction in Section 3, a partial normal form of order $m$ can be achieved by a sequence of merely $m+2$ {\em polynomial} transformations, where $k'$th transformation has  (weighted) degree $\leq k$. Calculating each of the transformations amounts to solving a (finite-dimensional) system of linear equations. A partial normal form of order $m$ allows to associate with two given near the origin ODEs $\mathcal E,\widetilde {\mathcal E}$ two {\em explicit} algebraic manifolds $\mathcal M_m,\widetilde{\mathcal M}_m$. If the manifolds $\mathcal M_m,\widetilde{\mathcal M}_m$ are distinct for some $m\geq 5$, we immediately conclude that the given ODEs  $\mathcal E,\widetilde {\mathcal E}$ are {\em non-equivalent} at the origin. In this way, we obtain an {\em effective} algorithm for distinguishing two given ODEs. For further details, we refer to the discussion in the end of Section 3. 
\end{remark}

An immediate intent, in view of \autoref{important}, is as follows.
\begin{problem}\label{painleve}
Compute a (partial) normal form \eqref{normalODE} for the class of Painlev\'e equations considered at a nonsingular point. 
\end{problem}
As discussed above,  solving \autoref{painleve} will imply a complete point classification of Painlev\'e equations. It will be also interesting to find applications of the normal form \eqref{normalODE} for Painlev\'e equations to {\em qualitative} problems related to them (for good expositions on the latter subject see, e.g., Clarkson \cite{clarkson}, Hinkannen and Laine \cite{hl}, Gromak, Laine and Shimomura \cite{gromak}).

\smallskip

\subsection{Overview of the method} The normal form construction given in the paper is inspired by the fundamental connection (dictionary) between Differential Equations and {\em CR-geometry}. Such a connection was first observed in the work of E.\,Cartan \cite{cartanODE,cartan} and Segre \cite{segre}, and was more recently revisited in a series of important publications due to Sukhov \cite{sukhov1,sukhov2}. The approach is based on a certain similarity between real submanifolds in complex space and {\em manifolds of solutions} of completely integrable systems of PDEs.  Application of Associated Differential Equations to studying CR-manifolds has recently led to important developments in CR-geometry (see, e.g., \cite{divergence},\cite{nonminimalODE},\cite{nonanalytic,analytic}). The present paper is probably the first work on applying the connection between CR-manifolds and Differential Equations in the other direction. Our construction  at a glance looks as follows. 

First, we are able to {\em reformulate}  the classification problem for ODEs in a totally different geometric language, by reducing it to the classification problem for so-called {\em manifolds of solutions} of ODEs, as described in Section 2. In Section 3, we study the geometry of manifolds of solutions and provide a {\em formal normal form} for them. For constructing the normal form, we use a development of the homological method due to Poincar\'e and Moser \cite{poincare}, \cite{chern}. In Sections 4 and 5, we prove the convergence of the normal form by employing a certain development of the well known in CR-geometry {\em Chern-Moser's chains}.    The final normalization result for manifolds of solutions is stated in \autoref{convergent}.

We were recently informed on an alternative approach to classifying second order ODEs in a paper in preparation due to A.\,Ottazzi and G.\,Schmalz \cite{os} (see also their earlier work \cite{os1}).

\bigskip

\begin{center}\bf
Acknowledgement 
\end{center}

\bigskip

The authors are sincerely grateful to Sergei Yakovenko for multiple useful discussions held during the preparation of the paper and his valuable suggestions. We also thank Boris Kruglikov for his useful comments on the text. The first author is supported by the Austrian Science Fund (FWF).
   



\section{Manifolds of solutions} Let \eqref{ODE}
be a second order ODE with the right hand side $F$ being defined near a fixed point $(x_0,y_0,u_0)\in J^1(\RR{},\RR{})$.
(Fixing $u_0$ amounts to fixing a distinguished solution for \eqref{ODE}, as discussed above). Let us  choose some parameterization 
$$(a,b)\mapsto \bigl\{y=\Phi(x,a,b)\bigr\}$$
of the space of solutions near the distinguished solution given by a diffeomorphism
\begin{equation}\label{psi}
	\psi=(\psi_1,\psi_2) 
	\quad \psi(a,b)=\left(y\rq{}(x_0),y(x_0)\right),
\end{equation}
sending abstract parameters $(a,b)$ 
into the standard ones that locally determine a solution.
 The associated submanifold
\begin{equation}\label{manifold}
\mathcal M=\bigl\{y=\Phi(x,a,b)\bigr\}
\end{equation}
of $\RR{4}$ is called \emph{an associated manifold of solutions} or \emph{an associated foliated submanifold} (note that this manifold is not uniquely associated with the initial ODE; any diffeomorphism in $a,b$ is allowed). The reason for the latter term comes from the fact that $\mathcal M$ possesses two natural foliations: $\mathcal S$ given by $a=const,\,b=const$ and corresponding to the solutions of \eqref{ODE}, and $\mathcal S^*$ given by $x=const,\,y=const$  and corresponding to the \emph{dual solutions} in the terminology of Cartan \cite{cartanODE}. Both foliations are invariant under the product group $\mbox{Diff}(\RR{2},0)\times\mbox{Diff}(\RR{2},0)$  consisting of transformations of the kind
\begin{equation}\label{coupled}
 \bigl(f(x,y),g(x,y),\lambda(a,b),\mu(a,b)\bigr)
 \end{equation}
 between manifolds of solutions. These two foliations enable us to define canonical $1$-dimensional subbundles $T^{\mathcal S}\mathcal M,\,T^{\mathcal S^*}\mathcal M$ of the tangent bundle $T\mathcal M$, fibers of which at any point $p=(x_0,y_0,a_0,b_0)\in\mathcal M$ are spanned  by the tangent vectors to the leaves  through the point $p$ of the foliations $\mathcal S$ and $\mathcal S^*$, respectively. Both subbundles are obviously invariant under transformations \eqref{coupled}. They also enable us to define 
   a canonical $2$-dimensional subbundle 
   $$T^C\mathcal M:=T^{\mathcal S}\mathcal M\oplus T^{\mathcal S^*}\mathcal M\subset T\mathcal M.$$ 
   It is immediate then that the subbundle $T^C\mathcal M$ is also invariant under transformations \eqref{coupled}. It is not difficult to see that coupled transformations \eqref{coupled} are precisely the ones preserving the structure
   $$T^C\mathcal M=T^{\mathcal S}\mathcal M\oplus T^{\mathcal S^*}\mathcal M\subset T\mathcal M.$$ 
   Using this structure, we can define on $\mathcal M$ a (nondegenerate!) canonical bilinear form 
   \begin{equation}\label{bilinear}
   \phi:\quad T^{\mathcal S}\mathcal M\times T^{\mathcal S^*}\mathcal M\lr T^M \big/T^C\mathcal M,
   \end{equation}
   given by Lie brackets of vector fields (local sections) $X\in T^{\mathcal S}\mathcal M,\,Y\in T^{\mathcal S^*}\mathcal M$. Its nondegeneracy can be easily checked in the coordinates \eqref{approximation} below.
   
\begin{remark}\label{globalparam}  
 We shall note that the mapping
   \begin{equation}\label{ME}
\chi:\,\,(x,a,b)\mapsto (x,\Phi(x,a,b),\Phi_x(x,a,b))
\end{equation}
establishes a natural diffeomorphism between a manifold of solutions and the $1$-jet space $J^1(\RR{},\RR{})$. Using this diffeomorphism, we can consider, for a fixed ODE $\mathcal E$, germs of the same manifold $\mathcal M$ at all possible points $p\in\mathcal M$  as manifolds of solutions for $\mathcal E$ near the target points $\chi(p)$. This observation will be useful for constructing later in this section global distinguished curves in $\mathcal M$ called distinguished chains.
\end{remark}

 We now choose  certain convenient coordinates for a manifold of solutions in which the associated geometric data looks maximally simple at a reference point $p=(x_0,y_0,a_0,b_0)\in\mathcal M$. First, we perform  a  quadratic change of coordinates  $x\mapsto x-x_0,\,y\mapsto y+Q(x)$ after which the distinguished point $(x_0,y_0,u_0)$ becomes the origin and the distinguished solution satisfies $y\rq{}\rq{}(0)=0$. Hence we have in \eqref{ODE} $F(0,0,0)=0$. Next, we require that an above parameterization map $\psi$ satisfies
   $$\psi=(\psi_1,\psi_2):\, (\RR{2},0)\mapsto(\RR{2},0),\quad d\psi(0,0)=\mbox{Id}, \quad (\psi_2)_{aa}(0,0)=0$$
   (this can be achieved by a quadratic transformation in $(a,b)$). We call all parameterizations with these properties \emph{admissible}.

A fundamentally important corollary of these preparations is as follows. Let us introduce for the coordinates of the space $\RR{4}_{x,y,a,b}$ the weights
\begin{equation}\label{weights}
[x]=[a]=1,\quad [y]=[b]=2.
\end{equation}
Then the submanifold \eqref{manifold} can be written as 

\begin{equation}\label{approximation}
\mathcal M=\Bigl\{y=b+ax +O(3)\Bigr\},
\end{equation}

\medskip
\noindent where we denote by $O(m)$ a function all Taylor polynomials of which have weights $\geq m$. Thus, 

\medskip

{\em all manifolds of solutions, as introduced above, admit the model}

\begin{equation}\label{model}
\mathcal Q=\bigl\{y=b+ax\bigr\} ,
\end{equation}
homogeneous with respect to the above choice of weights. The quadric $\mathcal Q$ is a manifold of solutions for the flat equation \eqref{flat}. We shall note that for a manifold of solutions \eqref{approximation} the fiber at the origin of the subbundle $T^C\mathcal M$ looks as
\begin{equation}\label{Pi}
\Pi=\bigl\{y= 0,\,\,b=0\bigr\}
\end{equation}
and the canonical bilinear form \eqref{bilinear} looks as
\begin{equation}\label{L}
\phi(x,a)=ax.
\end{equation}

 Let now $\mathcal G$ denotes the group of diffeomorphisms of $(\RR{2},0)$, $2$-jet prolongations of which preserve the  origin in 
  $J^2(\RR{},\RR{})$.
 Fix an ODE $\mathcal E$ as above and some $(f,g)\in\mathcal G$. Let $\tilde{\mathcal E}$ denotes the transformed (target) ODE, and $\tilde{\mathcal M}$ an associated manifold of solutions (with some admissible parameterization). Then it is immediate that there exists a unique $\bigl(\lambda(a,b),\mu(a,b)\bigr)\in\mathcal G$  such that the corresponding  diffeomorhism
\eqref{coupled}
 of $(\RR{4},0)$ onto itself transforms $(\mathcal M,0)$ into $(\tilde{\mathcal M},0)$. Conversely, if a coupled diffeomorphism \eqref{coupled} transforms $(\mathcal M,0)$ into $(\tilde{\mathcal M},0)$, then the map $\bigl(f(x,y),g(x,y)\bigr)$ transforms  $\mathcal E$ into $\tilde{\mathcal E}$. 
 
 We summarize our arguments in the following proposition.
 
\begin{proposition}\label{1to1}
 Transformations of two ODEs \eqref{ODE}  defined near the origin  with $F(0,0,0)=0$ are in 1-to-1 correspondence with product transformations from $\mathcal G\times \mathcal G$ of associated manifolds of solutions. 
\end{proposition}

 \smallskip

\section{Formal normal form for manifolds of solutions} According to the outcome of the  previous section, we approach the normal form problem for second order ODEs by constructing a normal form for the class of submanifolds \eqref{approximation} with respect to the group $\mathcal G\times \mathcal G$. For the latter problem, we start by studying the \emph{formal} version of this problem, that is, we replace the $C^\infty$ manifolds under consideration  by corresponding \emph{formal} submanifolds of $\RR{4}$, and the group of local $C^\infty$ diffeomorphisms $\mathcal G$ by the group $\mathcal G^f$ of \emph{formal} diffeomorphisms of $(\RR{2},0)$, $2$-jet prolongation of which preserves the origin in the jet space.

We first note that the model manifold $\mathcal Q$ (as in \eqref{model}) is invariant under the group of projective transformations from $\mathcal G^f\times\mathcal G^f$. The group of such transformations can be written as \eqref{G} and is generated by the subgroups
\begin{equation}\label{g0}
x\mapsto sx,\quad y\mapsto sty, \quad a\mapsto ta,\quad b\mapsto stb,\quad s,t\in\RR{},\quad s,t,\neq 0
\end{equation} 
and
\begin{equation}\label{g+}
\begin{aligned}
x\mapsto & \frac{x+\alpha y}{1-2\beta x +(r+\alpha\beta)y}\quad y\mapsto \frac{y}{1-2\beta x +(r+\alpha\beta)y}, \\ a\mapsto & \frac{a+2\beta b}{1+\alpha a+(r+3\alpha\beta) b},\quad b\mapsto \frac{b}{1+\alpha a+(r+3\alpha\beta) b},\quad \alpha,\beta,r\in\RR{}
\end{aligned}
\end{equation} 
(it is easy to see from a straightforward computation). In fact, the first two components in \eqref{g0},\eqref{g+} generate the entire group of projective transformations of $\RR{2}$, preserving the origin and the line $y=0$. We will prove later that \eqref{g0},\eqref{g+} generate the whole group of formal invertible transformations in $\mathcal G^f\times\mathcal G^f$, preserving the quadric $\mathcal Q$.

Let now $\mathcal M=\bigl\{y=\Phi(x,a,b)\bigr\}$ and  $\mathcal M^*=\bigl\{y=\Phi^*(x,a,b)\bigr\}$ be any two formal submanifolds of $\RR{4}$, satisfying \eqref{approximation}. The fact that a formal transformation \eqref{coupled} transforms $\mathcal M$ into $\mathcal M^*$ amounts to
\begin{equation}\label{identity}
g(x,y)=\Phi^*(f(x,y),\lambda(a,b),\mu(a,b))|_{y=\Phi(x,a,b)}.
\end{equation}

Let us consider expansions of $f,g,\lambda,\mu,\Phi,\Phi^*$ in weighted homogeneous polynomials $f_k(x,y),g_k(x,y),\lambda_k(a,b),\mu_k(a,b),\Phi_k(x,a,b),\Phi_k^*(x,a,b)$ with respect to the above choice of weights. We have
$$f_0=g_0=0,\quad \lambda_0=\mu_0=0$$ 
(since our map preserves the origin). In addition, \eqref{approximation} gives 
$$\Phi_0=\Phi_1=0,\,\Phi_2=b+ax.$$
 Considering then in \eqref{identity} terms of weight $1$ we get also 
 $$g_1=0,\quad\mu_1=0.$$
 
 We next collect in \eqref{identity} (which we consider as an identity of power series in $x,a,b$) all terms of weight $2$. This gives (by comparing all possible monomials of weight $2$)
$$f_1=sx,\quad \lambda_1=ta,\quad g_2=sty,\quad \mu_2=stb,\quad s,t\in\RR{},\quad s,t,\neq 0$$
 (the last restriction follows from the invertibility of the map). It is not difficult to see from here that any map $H$, as in  \eqref{coupled}, can be uniquely factored as 
 \begin{equation}\label{factor}
 H=\tilde H \circ H_0,
 \end{equation}
 where $H_0$  belongs to the group generated by \eqref{g0},\eqref{g+}, and $\tilde H$ is normalized by
\begin{equation}\label{normalmap}
\begin{aligned}
f_x(0,0)=g_y(0,0)=1,\quad f_{y}(0,0)=f_{xx}(0,0)=g_{yy}(0,0)=0
\end{aligned}
\end{equation}
(recall that $g_x(0,0)=0$ holds for all maps from $\mathcal G^f\times\mathcal G^f$, as they preserve the origin in $J^1(\RR{},\RR{})$).
Hence, 
\bigskip

{\it in what follows for the normalization procedure we may consider only maps satisfying \eqref{normalmap}}.

\bigskip

 We denote the subgroup in  $\mathcal G^f\times\mathcal G^f$ normalized by \eqref{normalmap} by $\mathcal F$. We then consider for $H\in\mathcal F$ in \eqref{identity} all terms of a fixed weight $m\geq 3$.
 Then, in view of 
 $$f=x+O(2),\quad g=y+O(3),\quad \lambda=a+O(2),\quad \mu=b+O(3),$$ 
 we obtain:  
 
 \begin{equation}\label{mm}
 g_{m}(x,b+ax)-\mu_m(a,b)-af_{m-1}(x,b+ax)-x\lambda_{m-1}(a,b)=\Phi_m^*(x,a,b)-\Phi_m(x,a,b)+\dots,
 \end{equation}
 
 \medskip
 
 \noindent where dots stand for a polynomial expression in  $x,a,b$, and also in all $f_{j-1},g_j,\lambda_{j-1},\mu_j$ with $j<m$ and their derivatives up to order $m$. To investigate solvability of the equations \eqref{mm},  introduce the linear operator 
 $$\mathcal L\bigl(f(x,y),g(x,y),\lambda(a,b),\mu(a,b)\bigr):=g(x,b+ax)-\mu(a,b)-af(x,b+ax)-x\lambda(a,b),$$
 mapping the linear space 
 $$\mathcal F_3=\left\{H-\mbox{Id},\,H\in\mathcal F\right\}$$
   onto some linear subspace $\mathcal V$ in the space $\mathcal T_3$ of all formal power series in $x,a,b$ of the form $O(3)$. Note that the left hand side in \eqref{mm} has the form $\mathcal L(f_{m-1},g_m,\lambda_{m-1},\mu_m)$. If $\mathcal N$ denotes now any linear subspace in $\mathcal T_3$ such that 
   \begin{equation}\label{decompose}
\mathcal T_3=\mathcal V\oplus\mathcal N
\end{equation}
(a transverse subspace), then the requirement $\Phi_m^*\in\mathcal N$ enables one to solve each of the equations in \eqref{mm} uniquely in $(f_{m-1},g_m,\lambda_{m-1},\mu_m)$, provided all $(f_{j-1},g_j,\lambda_{j-1},\mu_j)$ with $j<m$ are already determined. Thus, in order to obtain a formal normal form for all manifolds of solutions as above, one needs only to choose an appropriate subspace $\mathcal N\subset\mathcal T_3$ with \eqref{decompose}. 

Let us use for elements $\Phi(x,a,b)\in\mathcal T_3$ expansions of the form
$$\Phi(x,a,b)=\sum_{k,l\geq 0}\Phi_{kl}(b)x^ka^l.$$ 
We then   introduce the subspace $\mathcal N\subset\mathcal T_3$ defined by the conditions
\begin{equation}\label{normalform}
\begin{aligned}
\Phi_{k0}=0,\,k\geq 0,\quad\Phi_{0l}=0,\,l\geq 1,&\quad  \Phi_{k1}=0,\,k\geq 2,\quad\Phi_{1l}=0,\,l\geq 1,\\
\Phi_{22}=\Phi_{23}=&\Phi_{32}=\Phi_{33}=0.
\end{aligned}
\end{equation}

\begin{proposition}\label{unique}
For the subspace $\mathcal N\subset\mathcal T_3$ defined by the conditions \eqref{normalform} one has the decomposition \eqref{decompose}. 
\end{proposition}

\begin{proof} The statement of the proposition is equivalent to the fact that an
equation 
\begin{equation}\label{cm-eq}
\mathcal L(f,g,\lambda,\mu)=\Psi(x,y,a),\quad(f,g,\lambda,\mu)\in\mathcal F_3,\quad \Psi\in\mathcal T_3
\end{equation}
 in $(f,g,\lambda,\mu)$
has a unique solution, modulo $\mathcal N$ in the right-hand side,
for any fixed $\Psi$. Accordingly, we proceed with solving this equation modulo $\mathcal N$ in the right-hand side (that is, modulo the normalization condition \eqref{normalform}). We use expansions of the form
$$f(x,b+ax)=f(x,b)+f_y(x,b)ax+\frac{1}{2}f_{yy}(x,b)x^2a^2+\cdot\cdot\cdot,$$
and similarly for $g$. 
Substituting into \eqref{cm-eq} we get
the equation
\begin{multline} \label{cm-eq1} 
\left(g(x,b)+g_y(x,b)ax+\frac{1}{2}g_{yy}(x,b)x^2a^2+\frac{1}{6}g_{yyy}(x,b)x^3a^3+\cdot\cdot\cdot\right)-\\
 -a\left(f(x,b)+f_y(x,b)ax+\frac{1}{2}f_{yy}(x,b)x^2a^2+\cdot\cdot\cdot\right)-\\
- x\lambda(a,b)-\mu(a,b)=\Psi(x,a,b).
\end{multline}
For the sequel of the proof, we expand $f(x,y)$ as $f=\sum\limits_{k\geq 0}f_k(y)x^k,$ and
similarly for $g$, and also $\lambda(a,b)$ as $\lambda=\sum\limits_{k\geq 0}\lambda_k(b)a^k,$ and
similarly for $\mu$.    Collecting then in \eqref{cm-eq1} all terms with $x^0a^lb^m,\,l,m\geq 0$, we obtain
\begin{equation} \label{0l} 
g(0,b)-\mu(a,b)=\Psi(0,a,b).
\end{equation}
The equation \eqref{0l} enables us to determine uniquely the functions $\mu_l(b),\,l\geq 1$ from the data $\Psi_{0l},\,l\geq 1$. In addition, we have 
\begin{equation}\label{mu0}
\mu_0(b)=g_0(b)-\Psi_{00}(b).
\end{equation} 
Collecting then terms with $x^ka^0b^m,\,k,m\geq 0$, we obtain
\begin{equation}\label{k0}
g(x,b)-x\lambda(0,b)-\mu(0,b)=\Psi(x,0,b).
\end{equation}
The equation \eqref{k0} enables us to determine uniquely all $g_k(b),\,k\geq 2$ from the data $\Psi_{k 0},\,k\geq 1$.  In addition, we have the condition
\begin{equation}\label{g1}
g_1=\lambda_0+\Psi_{10}.
\end{equation}
Next, we collect terms with $x^1a^lb^m,\,l,m\geq 0$ and get
\begin{equation}\label{1l}
g_1(b)x+g_0\rq{}(b)ax-af_1(b)x-a^2f_0\rq{}(b)x-x\lambda(a,b)=\sum_{l\geq 0}\Psi_{1l}(b)xa^l.
\end{equation}
The equation \eqref{1l} enables us to determine uniquely all $\lambda_l(b),\,l\geq 3$ from the data  $\Psi_{1l},\,l\geq 3$.  In addition, we have 
\begin{equation}\label{lambda12}
\lambda_1=g_0\rq{}-f_1-\Psi_{11},\quad \lambda_2=-f_0\rq{}-\Psi_{12}.
\end{equation}
For the terms with $x^ka^1b^m,\,k,m\geq 0$ we obtain
\begin{equation}\label{k1}
g_y(x,b)ax-af(x,b)-x\lambda_1(b)a-\mu_1(b)a=\sum_{k\geq0}\Psi_{k1}(b)x^ka
\end{equation}
The equation \eqref{k1} enables us to determine uniquely all $f_k(b),\,k\geq 3$ from the data  $\Psi_{k 1},\,k\geq 3$.  In addition, we have 
\begin{equation}\label{f2}
f_2=g_1\rq{}-\Psi_{21}.
\end{equation}
Now we consider all terms with $x^2a^3b^m,\,m\geq 0$, and this gives
\begin{equation}\label{f0}
f_0\rq{}\rq{}=-2\Psi_{23}.
\end{equation}
In view of \eqref{normalmap}, \eqref{f0} enables us to determine uniquely the function $f_0$ from the data $\Psi_{23}(b)$. For the terms $x^3a^2b^m,\,m\geq 0$ we have
\begin{equation}\label{lambda0}
\frac{1}{2}g_1\rq{}\rq{}-f_2\rq{}=\Psi_{32}.
\end{equation}
In view of the conditions \eqref{normalmap}, the equations \eqref{g1},\eqref{f2} and \eqref{lambda0} enable us to determine uniquely the functions $f_2,g_1,\lambda_0$ from the data $\Psi_{32},\Psi_{10},\Psi_{21}$. 

It remains to consider the terms $x^2a^2b^m,\,m\geq 0$ and $x^3a^3b^m,\,m\geq 0$. For these ones we get the system:
\begin{equation}\label{system}
\begin{aligned}
\frac{1}{2}g_0\rq{}\rq{}-f_1\rq{}=\Psi_{22},\\
\frac{1}{6}g_0\rq{}\rq{}\rq{}-\frac{1}{2}f_1\rq{}\rq{}=\Psi_{33}.
\end{aligned}
\end{equation}
The system \eqref{system} enables us to determine uniquely the functions $f_1,g_0,\lambda_0$ from the data $\Psi_{22},\Psi_{33}$ (in view of \eqref{normalmap}). After that, the equations \eqref{mu0} and \eqref{lambda12} enable us to determine uniquely the functions $\mu_0,\lambda_1,\lambda_2$ from the data $\Psi_{00},\Psi_{11},\Psi_{12}$.

All the terms, concerned in the normalization conditions \eqref{normalform}, have been considered, and this proves the proposition.
\end{proof}

By this we have proved the main result of this section.

\begin{proposition}\label{formal}
For any formal manifold of solutions $\mathcal M$, as in \eqref{approximation}, and a tuple of real parameters $(s,t,\alpha,\beta,r)$, $s,t,\neq 0$, there exists a unique formal  transformation from $\mathcal G^f\times\mathcal G^f$ with 
\begin{equation}\label{parameters}
f_x(0)=s,\quad   g_y(0)=t,\quad f_y(0)=\alpha,\quad f_{xx}(0)=\beta,\quad g_{yy}(0)=r
\end{equation}
bringing $\mathcal M$ to a manifold of solutions in the normal form \eqref{normalform}.
\end{proposition}

As an application, we obtain the description of the automorphism group of the model \eqref{model}.

\begin{corollary}
The group of formal invertible  transformations  from $\mathcal G^f\times\mathcal G^f$, mapping the quadric \eqref{model} onto itself, is generated by the projective transformations \eqref{g0},\eqref{g+}. 
\end{corollary}

Further, at a strongly nonflat point, we may obtain a unique normal form, up to a simple linear action of $\mathbb Z_2$ on normal forms.

\begin{proposition}\label{uptoZ2}
If a formal manifold of solution $\mathcal M$ is associated with a the germ at the origin of a (formal) ODE $\mathcal E$, as in \eqref{ODE}, and the origin is a strongly nonflat point for $\mathcal E$, then $\mathcal M$ can be brought to a special normal form, given by the conditions \eqref{normalform} supplemented by
\begin{equation}\label{snm}
\Phi_{42}(0)=\Phi_{24}(0)=1,\quad \Phi_{43}(0)=\Phi_{34}(0)=\Phi_{53}(0)=0.
\end{equation}
A special normal form \eqref{snm} for $\mathcal M$ is unique, up to the linear action
\begin{equation} \label{Z2decoupled}
x\mapsto -x, \quad y\mapsto y, \quad a\mapsto -a, \quad b\mapsto b
\end{equation}
of $\mathbb Z_2$ on the special normal forms in the real case, and up to the linear action
\begin{equation} \label{Z8decoupled}
x\mapsto \varepsilon x, \quad y\mapsto y, \quad a\mapsto \varepsilon^{-3}a, \quad b\mapsto b, \quad \varepsilon^8=1
\end{equation} 
in the complex case.
\end{proposition} 

\begin{proof}
Let us bring $\mathcal M$ to some normal form \eqref{normalform}. Let us note that if an ODE $\mathcal E$, as in \eqref{ODE}, corresponds to a manifold of solutions  in the normal form \eqref{normalform}, then a straightforward computation gives  
\begin{equation}\label{I1I2}
I_1(F)(0,0,0)=\frac{\partial^4 F}{\partial y\rq{}^4}(0,0,0)=48\Phi_{42}(0),\quad I_2(F)(0,0,0)=\frac{\partial^4 F}{\partial x^2\partial y\rq{}^2}(0,0,0)=48\Phi_{24}(0).
\end{equation}
Thus, at a strongly nonflat point we have:
\begin{equation}\label{nonflat}
\Phi_{42}(0)\neq 0, \quad \Phi_{24}(0)\neq 0.
\end{equation}
We next consider all transformations, bringing $\mathcal M$ to a normal form.
In what follows in the proof we switch to the notations in \eqref{g0},\eqref{g+} for fixing parameters of a normalizing map (not to be confused from the ones in \eqref{parameters}!).
It is easy to see that by a scaling \eqref{g0}, $\Phi_{42}(0),\Phi_{24}(0)$ change as $$\Phi_{42}(0)\mapsto s^3t\Phi_{42},\quad \Phi_{24}(0)\mapsto st^3\Phi_{24}.$$
Thus, we can achieve by a scaling \eqref{g0} 
\begin{equation}\label{specified}
\Phi_{42}(0)=  \Phi_{24}(0)=1,
\end{equation} 
and the respective parameters $t,s$ are defined uniquely up to a choice of $\varepsilon$ with $\varepsilon^8=1$ (which gives rise to $\varepsilon=\pm 1$ in the real case). Thus, once \eqref{specified} is achieved, we may assume, up to the actions \eqref{Z2decoupled},\eqref{Z8decoupled} respectively, the parameters $s,t$ in \eqref{g0} to be equal to $1$. 

Next, we have to consider the action of the remaining parameters $\alpha,\beta,r$ in \eqref{parameters}    on the space of normal forms \eqref{normalform} of a fixed manifold of solutions. We consider the basic identity \eqref{identity} for two manifolds $\mathcal M,\mathcal M^*$ in the normal form \eqref{normalform} and a map $(f,g,\lambda,\mu)$ between them with the identity parameters $s,t$. Using the fact that both $\mathcal M, \mathcal M^*$ coincide with the model \eqref{model} up to weight $5$, we first conclude that the components $(f_{j-1},g_j,\lambda_{j-1},\mu_j)$ with $j\leq 6$ coincide with that for a map \eqref{g+} with the same parameters $\alpha,\beta,r$. It is straightforward to compute then that for the desired target Taylor coefficients $\Phi^*_{43}(0),\Phi^*_{34}(0),\Phi^*_{53}(0)$ we have:
\begin{equation}\label{changePhi}
\Phi^*_{43}(0)=\Phi_{43}(0) + \alpha, \quad \Phi^*_{34}(0)=\Phi_{34}(0)-2\beta, \quad \Phi^*_{53}(0)=\Phi_{53}(0)+3r+\tau(\alpha,\beta),
\end{equation}
where $\tau(\alpha,\beta)$ is a specific polynomial of $\alpha,\beta$ exact form of which is of no interest to us. Now formulas \eqref{changePhi} immediately imply that for a fixed source $\mathcal M$  there exists precisely one collection $(\alpha,\beta,r)$ giving the desired conditions \eqref{snm}, which proves the proposition.

\end{proof}

\smallskip

We shall now provide an important discussion of the formal normalization procedure described in this section. Most importantly, this procedure  allows not only for obtaining a {\em complete} set of invariants of an ODE (as we will see from the convergence theorem below), but also allows for a {\em partial} normal form (normal form to weight $m$), which can be obtained by a {\em polynomial transformation}, finding exact coefficients of which amounts to solving a (finite-dimensional) system of linear equations. 

Indeed, let us consider a germ at the origin of an ODE $\mathcal E$, as in  \eqref{ODE}. For simplicity, assume that the origin is not a strongly-nonflat point of it, and fix an integer $m\geq 8$. By substituting $y=\Phi(x,a,b)$ satisfying $y(0)=b,\,y'(0)=a$,   we successively find as many Taylor coefficients of $\Phi$ as desired. In this way, we find all the weighted homogeneous polynomials $\Phi_j(x,a,b)$ in the above notations, $1\leq j\leq m$. Following then the successive procedure described above,  we successively solve all the equations \eqref{mm} modulo the space $\mathcal N$ to find the $\Phi^*_j,\,1\leq j\leq m$ corresponding to the normal form (for each $j$, this amounts to solving a system of linear equations). Adapting the procedure to the proof of \autoref{uptoZ2}, we end up with a {\em polynomial} manifold of solutions (terms in the expansion of $\Phi$ go up to weight $m$) which is in the {partial normal form} \eqref{sn} up to weight $m$ (that is, the manifold of solutions coincides with a one in the special normal form \eqref{sn} up to weight $m$).  

Now, in case the latter procedure for two given ODEs gives {\em different} results (i.e. the polynomials standing in the partial normal form are different for the two manifolds of solutions), we conclude that the two given germs of ODEs are inequivalent. As discussed in the Introduction, the later procedure can be applied for the point classification of Painlev\'e equations (with possibly using the MAPLE package).


\section{Canonical curves in a manifold of solutions} We now approach to the problem of \emph{convergence} of a transformation, bringing a manifold of solutions to a normal form, by introducing below a crucial for the geometry of manifolds of solutions geometric objects that we call \emph{chains}.  

Let us start with a smooth manifold of solutions $\mathcal M=\bigl\{y=\Phi(x,a,b)\bigr\}$, associated with the initial ODE \eqref{ODE}, and consider a point $p=(x_0,y_0,a_0,b_0)$ on it. (Using \autoref{globalparam} we consider $\mathcal M$ as globally associated with $\mathcal E$). We may then consider  the corresponding formal manifold $\widehat{\mathcal M}=\bigl\{y=\widehat{\Phi}(x,a,b)\bigr\}$ through $p$, and apply to $\widehat{\mathcal M}$ the formal normalization procedure from the previous section.   Recall that, in any coordinates \eqref{approximation}, a formal normalizing transformation $H$ is unique up to a choice of an element $H_0$ of the transformation group  generated by \eqref{g0},\eqref{g+}. Let us fix any direction $d$ in  the tangent space $T_p\mathcal M$ which is transverse to the distinguished plane $T^C_p\mathcal M$, and take a formal normalizing transformation $H$ such that the image $dH|_p(d)$ is  the vertical line
\begin{equation}\label{Gamma}
\Gamma=\bigl\{x=a=0,\quad y=b\bigr\}
\end{equation}
(note that $\Gamma\subset \mathcal M_N$, where $\mathcal M_N=H(\mathcal M)$ is the formal normal form). Note that such a transformation $H$ is not unique, as a choice of $d$ corresponds to a choice of parameters $\alpha,\beta$ in \eqref{g+}, while the other parameters may vary.

We now consider $(x,a,b)$ as local coordinates on $\mathcal M$ and introduce a bundle  $\mathcal M^{(1)}$ over $\mathcal M$ which is  the  bundle of $1$-jets of curves of the form 
$$\gamma=\bigl\{x=x(b),\, a=a(b)\bigr\}\subset\mathcal M.$$ 
$\mathcal M^{(1)}$ has the local coordinates $(x,a,b,u_1,u_2)$ (where the last two coordinates correspond to the derivatives $\frac{dx}{db},\frac{da}{db}$). Fixing a direction $d$ as above amounts to fixing a point  $p^{(1)}\in\mathcal M^{(1)}$ over $p$.  The choice of a direction given by $\eqref{Gamma}$ in $T_0\mathcal M_N$ corresponds in this way to the origin in the $1$-jet bundle $\mathcal M^{(1)}_N$. Let then $H^{(1)}$ denotes the lifting of $H$ to a map $(\mathcal M^{(1)},p^{(1)})\mapsto (\mathcal M^{(1)}_N,0)$. If now $e\in T_{p^{(1)}}\mathcal M^{(1)}_N$ is the vector given in local coordinates as $(0,0,1,0,0)$ (that is, $e$ is tangent to the lifting of \eqref{Gamma} to the  bundle $\mathcal M^{(1)}_N$), we define a direction field in $\mathcal M^{(1)}$ as follows:
\begin{equation}\label{direction}
l(p^{(1)}):=\Bigl(dH^{(1)}|_{p^{(1)}}\Bigr)^{-1}\bigl (e\bigr).
\end{equation}
It can be seen directly from the factorization \eqref{factor} and the formulas \eqref{g0},\eqref{g+} that the definition of $l(p^{(1)})$ does \emph{not} depend on the choice of a formal normalizing transformation $H$ with  $dH|_p(d)=\Gamma$ (that is, variation of parameters $s,t,r$ in \eqref{g0},\eqref{g+} does not change the direction $l(p^{(1)}$).

 \begin{proposition}\label{smooth}
 Formula \eqref{direction} defines a $C^\infty$ smooth direction field in $\mathcal M^{(1)}$.
 \end{proposition}
 
 \begin{proof}
We first note that the coordinates bringing $(\mathcal M,p)$ to the form \eqref{approximation} depend smoothly on $p\in\mathcal M$ (as the coordinate change is quadratic with coefficients depending on the $2$-jet of $\Phi$ at $p$), and so do the coordinates in the $1$-jet bundle (in their dependence on $p^{(1)}$). Next, we claim that the desired
direction \eqref{direction}  can be also defined without using formal transformations.
Indeed, for a manifold \eqref{approximation} it follows from the normal form construction that, as soon
as the initial weighted polynomials
$\{\Phi^*_j,f_{j-1},g_j,\,3\leq j\leq m\}$ for some $m\geq 3$ have
been determined, they do not change after further normalization of
terms of higher weight. Hence, solving the equations \eqref{mm}
up to $m$ sufficiently large (namely, for all $m\leq 5$), we uniquely determine
$dH^{(1)}|_{p^{(1)}}$. 
It is also not difficult to see from here already that the constructed direction 
field is smooth. Indeed,  each fixed weighted polynomial $\Phi_m$,
 depends on $p$ smoothly (because the coordinates \eqref{approximation}  do). Hence all polynomials $f_m$ and $g_m$ depends on $p^{(1)}$
smoothly, as it is obtained by solving a system of linear
equations with a fixed nondegenerate matrix in the left-hand side and
right-hand side smooth in $p^{(1)}$ (the latter fact can be seen from
the proof of \autoref{formal}). We immediately conclude from here
$dH^{(1)}|_{p^{(1)}}$ depends on $p^{(1)}$ smoothly, and so does $l(p^{(1)})$.
 \end{proof}

We now integrate the smooth direction field  $l(p^{(1)})$ and obtain a canonical foliation in the $1$-jet bundle $\mathcal M^{(1)}$. Projecting the leaves of this foliation onto the base $\mathcal M$ gives us smooth curves in $\mathcal M$ that we call \emph{chains}. The canonical map \eqref{MtoE} from $\mathcal M$ onto the $1$-jet space $J^1(\RR{},\RR{})$ gives us a smooth $2$-parameter family of smooth curves, each of which we as well call \emph{a chain}. As follows from the constructions, there is exactly one chain in each transverse direction (to either the $2$-plane $T^C\mathcal M$ or the $2$-plane $\Pi$, as in \eqref{Pi}). 

It is not difficult to see that for a real-analytic manifold of solutions all the chains are real-analytic as well, as the direction field \eqref{direction} is real-analytic in this case. We summarize the section by the following proposition.
 
 \begin{proposition}\label{chains} Let $\mathcal E$ be a smooth (resp. real-analytic) second order ODE, as in \eqref{ODE}, defined near a point $(x_0,y_0,u_0)\in J^1(\RR{},\RR{})$, and $\mathcal M$ its manifold of solutions defined near the corresponding point $p=(x_0,y_0,a_0,b_0)\in\mathcal M$. Then
 
 \smallskip
 
 {\bf (a)} For any  direction $d\subset T_p\mathcal M$, transverse to the space $T^C_p \mathcal M$, there exists a unique chain $\tilde \gamma\subset\mathcal M$, passing through $p$ and tangent to $d$ at $p$. If $\mathcal M$ is real-analytic, then $\tilde \gamma$ is real-analytic as well.

\smallskip

{\bf (b)} The family of all chains in $\mathcal M$ is invariant under (real-analytic) product diffeomorphisms \eqref{coupled} of $\RR{4}$, while the family of all possible chains of the ODE $\mathcal E$ is invariant under (real-analytic) point diffeomorphisms of $J^1(\RR{},\RR{})$.
\end{proposition}

\section{Existence of a convergent normalizing transformation} In this section we show the existence of a $C^\infty$ smooth transformation bringing a smooth manifold of solutions  \eqref{approximation} to a normal form \eqref{normalform}, and apply this to the proof of \autoref{theor1}.  We do so in several steps, each of which has a certain geometric interpretation that we address below. In what follows for a smooth function $\Phi(x,a,b)$ we denote by $\Phi_{kl}(b)$ the partial derivatives 
$$\left.\frac{1}{k!l!}\frac{\partial^{k+l}\Phi}{\partial x^k\partial a^l}\right|_{x=a=0}$$
 (all of which are again $C^\infty$ smooth  functions).

\smallskip

\noindent{\bf 1. Normalization  of a chain.} For a manifold of solutions $\mathcal M$, as in \eqref{approximation}, we choose any direction $d\subset T_0\mathcal M$, transverse to the space $T^C_0 \mathcal M$, and consider the corresponding chain $\tilde \gamma\subset\mathcal M$, which is given by 
$$\tilde \gamma=\bigl\{x=f_0(b), \,\, a=\lambda_0(b),\,\, y=g_0(b)\bigr\},\quad f_0(0)=\lambda_0(0)=g_0(0)=0,\quad g_0\rq{}(0)\neq 0.$$ 
We then perform the transformation
\begin{equation}\label{kill00}
x=x^*+f_0(y^*),\quad y=g_0(y^*), \quad a=a^*+\lambda_0(b^*), \quad b=b^*.
\end{equation}
It is easy to see that the chain $\tilde\gamma$ is transformed into the vertical line $\Gamma$, as in  \eqref{Gamma}, by means of \eqref{kill00}, as desired. The form \eqref{approximation} is clearly preserved. The fact that the new manifold of solutions contains the line \eqref{Gamma} yields
$$\Phi^*_{00}(b)=0$$
for its defining function. In the sequel we consider only transformations, preserving the condition
\begin{equation}\label{Gammain}
\Gamma\subset\mathcal M.
\end{equation}

\smallskip

\noindent{\bf 2. Normalization of leaves of the canonical foliations $\mathcal S$ and $\mathcal S^*$ along the chain.} First,  we aim to straighten that leaves of the second canonical foliation $\mathcal S^*$ on $M$ which intersect the chain $\Gamma\subset\mathcal M$. For doing so, for a manifold of solutions obtained in the previous step we apply the transformation
\begin{equation}\label{kill0l}
x^*=x,\quad y^*=y,\quad a^*=a, \quad b^*=\Phi(0,a,b).
\end{equation} 
In view of \eqref{approximation}, this map is invertible. The conditions \eqref{approximation} and \eqref{Gammain} are preserved. The fact that the leaved are straightened along the chain means that $\Phi^*(0,a^*,b^*)=b^*$, while the latter condition is satisfied by applying the basic identity \eqref{identity} to the map \eqref{kill0l}. Note that for the new manifold of solutions we have
$$\Phi^*_{0l}(b)=0.$$

Second, we aim to straighten the leaves of the first canonical foliation $\mathcal S$ on $M$ intersecting the chain $\Gamma\subset\mathcal M$. For that, for a manifold of solutions obtained in the previous step we apply the transformation given by
\begin{equation}\label{killk0}
x=x^*,\quad y=g(x^*,y^*),\quad a=a^*, \quad b=b^*\quad\mbox{for}\quad g(x^*,y^*):=\Phi(x^*,0,y^*).
\end{equation} 
In view of \eqref{approximation}, this map is invertible. It is not difficult to see that all the previously achieved normalization conditions  are preserved. The fact that the leaves of $\mathcal S$ are straightened along the chain means that $\Phi^*(x^*,0,b^*)=b^*$. Substituting  the latter condition into the basic identity \eqref{identity} applied to the map \eqref{kill0l}, we see that it amounts to 
$$g(x^*,b^*)=\Phi(x^*,0,b^*),$$
as desired. Note that for the new manifold of solutions we have
$$\Phi^*_{k0}(b)=0.$$

\smallskip

\noindent{\bf 3. Fixing parameterizations of leaves of the canonical  foliations $\mathcal S$ and $\mathcal S^*$.}  Leaves of the foliation $\mathcal S$ are parameterized by the parameters $a_0,b_0$ as $s_{a_0,b_0}=\{y=\Phi(x,a_0,b_0)\}$, and that of the foliation $\mathcal S^*$ are parameterized as $s^*_{x_0,y_0}=\{y_0=\Phi(x_0,a,b)\}$. It is convenient to  solve this for $b$ and get
$s^*_{x_0,y_0}=\{b=\Phi^*(a,x_0,y_0)\}$. In this step we aim to fix these two parameterizations by the condition $(a_0,b_0)=\left(y\rq{}(0),y(0)\right)$ for the foliation $\mathcal S$ and by the condition $(x_0,y_0)=\left(b\rq{}(0),b(0)\right)$ for the foliation $\mathcal S^*$. Geometrically this means that each of the leaves of the foliation $\mathcal S$ is parameterized by its $1$-jet at the intersection point with the distinguished (thanks to the previous step) plane $\{x=0\}$, while each of the leaves of the foliation $\mathcal S^*$ is parameterized   by its $1$-jet at the intersection point with the distinguished plane $\{a=0\}$.

For the defining function $\Phi$ obtained in the previous step this amounts to the conditions
\begin{equation}\label{derivatives}
\Phi_x(0,a,b)=a,\quad\mbox{and}\quad \Phi_a(x,0,b)=x,
\end{equation}
respectively, which alternatively means 
$$\Phi_{1l}(b)=\Phi_{k1}(b)=0,\quad k,l\geq 1.$$

We first aim to achieve the condition $\Phi_{11}(b)=0$, which reads as $\Phi_{xa}(0,0,b)=1$. This can be interpreted as fixing the parameterization of the leaves of the canonical foliations first along the chain $\Gamma$. We do so by performing the transformation
\begin{equation}\label{kill11}
x^*=x,\quad y^*=y, \quad a^*=a\lambda_1(b)\quad b^*=b,
\end{equation}
where the function $\lambda_1(b)=1+O(b)$ will be determined later. 
Applying the basic identity \eqref{identity} to this map and using $\Phi^*_{x^*a^*}|_{x^*=a^*=0}=1$ we conclude, by differentiating in $x,a$ and substituting $x=a=0$, that $\Phi_{11}(b)=\lambda_1(b)$, which determines $\lambda_1$ with the desired property uniquely.

In order to achieve the first condition in \eqref{derivatives} we perform the  transformation 
\begin{equation}\label{kill1l}
x^*=x,\quad y^*=y,\quad a^*=\lambda(a,b), \quad b^*=b
\end{equation}
for an appropriate function $\lambda(a,b)$ with $\lambda=a+O(a^2)$ which will be determined later. Applying the basic identity \eqref{identity} to the map \eqref{kill1l}, we get
$$\Phi(x,a,b)=\Phi^*(x,\lambda(a,b),b).$$
Differentiating  in $x$ and substituting $x=0$, we get, in view of $\Phi^*_{x^*}(0,a^*,b^*)=a^*$:
$$\Phi_x(0,a,b)=\lambda(a,b),$$
which determines $\lambda$ with the desired property uniquely.

For the second condition in \eqref{derivatives}, we perform to the previously obtained manifold of solution the  transformation 
\begin{equation}\label{killk1}
x^*=f(x,y),\quad y^*=y,\quad a^*=a, \quad b^*=b
\end{equation}
for an appropriate function $f(x,y)$ with $f=x+O(x^2)$ which will be determined later. Note that such a transformation does not change the properties of $\mathcal M$ achieved in the previous steps. Applying then the basic identity \eqref{identity} to the map \eqref{killk1}, we get
$$\Phi(x,a,b)=\Phi^*(f(x,y),a,b)|_{y=\Phi(x,a,b,)}.$$
Differentiating  in $a$ and substituting $a=0$, we get, in view of the properties $\Phi^*_{x^*}(x^*,0,b)=0$,  $\Phi(x,0,b)=b$, and $\Phi^*_{a^*}(x^*,0,b^*)=x^*$:
$$\Phi_a(x,0,b)=f(x,b),$$
which determines $f(x,y)$ with the desired property uniquely.

We end up with a manifold of solutions of the form

\begin{equation}\label{prenormal}
y=b+ax+\sum_{k,l\geq 2}\Phi_{kl}(b)x^ka^l.
\end{equation}

\smallskip

\noindent{\bf 4. Fixing a normal basis for the canonical bilinear form along the chain.} 
As the next step in the normalization procedure we perform a transformation
\begin{equation}\label{kill22}
x^*=f(y)x,\quad y^*=y, \quad a^*=\frac{1}{f(b)}a,\quad b^*=b
\end{equation}
for an appropriate function $f(b)=1+O(b)$ that will be determined later in order to achieve the condition $\Phi^*_{xxaa}(0,0,b)=0$, or alternatively $\Phi^*_{22}(b)=0$. To explain the geometric meaning of this step we note that, as follows from \eqref{prenormal}, the canonical bilinear form \eqref{bilinear} is given by \eqref{L} along the chain $\Gamma$. Thus the transformation \eqref{kill22}, which does not change the bilinear form \eqref{bilinear} along $\Gamma$, can be interpreted as a choice of a \emph{normal basis} for the form \eqref{bilinear} at any point $p\in\Gamma$ (that is, a basis where this form is given by \eqref{L}).

The basic identity for the map \eqref{kill22} looks as 
$$\Phi(x,a,b)=\Phi^*\bigl(xf(\Phi(x,a,b)),a,b\bigr).$$
Differentiating both sides twice in $x$ and twice in $a$ and evaluating at $x=a=0$, it is not difficult to compute that the requirement $\Phi^*_{xxaa}(0,0,b)=0$ amounts to
\begin{equation}\label{find22}
\Phi_{xxaa}(0,0,b)= \frac{f\rq{}(b)}{f(b)}.
\end{equation}
Considering  \eqref{find22} as a first order ODE with the  initial data $f(0)=1$, we find $f$ with the desired properties uniquely. Thus for the new manifold of solutions we have
$$\Phi^*_{22}(b)=0.$$
However, one can say even more.

\begin{proposition}\label{32dies}
For the new manifold of solutions  we have, in addition, 
$$\Phi_{32}(b)=\Phi_{23}(b)=0.$$
\end{proposition}

\begin{proof}
We first claim that
$$\Phi_{32}(0)=\Phi_{23}(0)=0.$$
 Indeed, consider  a formal map $H=(f,g,\lambda,\mu)$ with $s=t=1$ in \eqref{mapdata}, which maps the formal manifold of solutions $\widehat{\mathcal M}$ corresponding to $(\mathcal M,0)$ into a formal normal form, and the vertical direction \eqref{Gamma} into itself. Since $\Gamma\subset\mathcal M$ is a chain, this map satisfies, by the definition of a chain,
\begin{equation}\label{usechain}
f_y(0,0)=f_{yy}(0,0)=\lambda_{bb}(0,0)=0.
\end{equation}
Let us then consider the basic identity \eqref{identity}, applied to the manifold $\widehat{\mathcal M}$, its normal form $\widehat{\mathcal M}_N$ and the map $H$. Note that the source and the target coincide with the model \eqref{model} up to weight $4$ (in view of \eqref{prenormal} and  $\Phi_{22}(b)=0$). Hence the weighted components $f_{j-1},g_j,\,j\leq 4$ coincide with that for \eqref{g+}; that is, they all vanish (in view of \eqref{usechain}). For the weight $m=5$ we then have
\begin{equation}\label{weight5}
 \mathcal L(f_4,g_5,\lambda_4,\mu_5)=-\Phi_{32}(0)x^3a^2-\Phi_{23}(0)x^2a^3
\end{equation}
(in the notations of \eqref{mm}).  Switching then to  the notations of the proof of \autoref{unique} and collecting in \eqref{weight5} terms $x^2a^1b^1,\,x^2a^3b^0,\,xb^2$, and $x^3a^2b^0$, it is not difficult to check that we obtain respectively the equations 
\begin{equation}\label{find32}
\begin{aligned}
f_2\rq{}(0)-g_1\rq{}\rq{}(0)&=0,\\
f_0\rq{}\rq{}(0)&=-2\Phi_{23}(0),\\
g_1\rq{}\rq{}(0)-\lambda_0\rq{}\rq{}(0)&=0,\\
\frac{1}{2}g_1\rq{}\rq{}(0)-f_2\rq{}(0)&=\Phi_{32}(0).
\end{aligned}
\end{equation}
 Since from \eqref{usechain} we have $\lambda_0\rq{}\rq{}(0)=f_0\rq{}\rq{}(0)=0$, the system \eqref{find32} gives $\Phi_{23}(0)=\Phi_{32}(0)=0$, which proves the claim.

Note then that because $\Gamma$ is a chain, and the prenormal form \eqref{prenormal} is invariant under shifts $b\mapsto b+b_0$, the same argument applies to any point $p\in\Gamma$. This finally proves 
$$\Phi_{23}(b)=\Phi_{23}(b)=0$$
in \eqref{prenormal}, as desired.
\end{proof}

\smallskip

\noindent{\bf 5. Fixing a parameterization for the chain.} It remains to achieve the last normalization condition $\Phi_{33}(b)=0$, or $\Phi_{x^3a^3}(0,0,b)=0$. We perform a transformation
\begin{equation}\label{kill33}
x^*=f(y)x,\quad y^*=g(y), \quad a^*=\lambda(b)a,\quad b^*=g(b)
\end{equation}
for  appropriate function $f,g,\lambda$ with $f(0),g\rq{}(0),\lambda(0)\neq 0$.  This transformation can be interpreted as a choice of a parameterization on the chain $\Gamma$. The basic identity \eqref{identity} looks as 
\begin{equation}\label{identity33}
g(\Phi(x,a,b))=\Phi^*\bigl(xf(\Phi(x,a,b)),a\lambda(b),g(b)\bigr).
\end{equation}
By differentiating \eqref{identity33} in $x,a$ and evaluating at $x=a=0$, it is not difficult to verify that for any $f,g,\lambda$ the function $\Phi^*$ satisfies
$$\Phi^*_{k1}=\Phi^*_{1l}=\Phi^*_{32}=\Phi^*_{23}=0,\quad k,l\geq 2.$$
Thus, in order to find $f,g,\lambda$ for which $\Phi^*$ has the desired form we need
\begin{equation}\label{conditions}
\Phi^*_{11}=\Phi^*_{22}=\Phi^*_{33}=0.
\end{equation}
Applying to \eqref{identity33} $\frac{\partial^2}{\partial x\partial a},\,\frac{\partial^4}{\partial x^2\partial a^2},\,\frac{\partial^6}{\partial x^3\partial a^3}$ and evaluating at $x=a=0$, we obtain respectively
\begin{equation}\label{functions}
\begin{aligned}
g\rq{}(b)&=f(b)\lambda(b),\\
2g\rq{}\rq{}(b)&=4f\rq{}(b)\lambda(b),\\
9g\rq{}(b)+\frac{3}{2}g\rq{}\rq{}\rq{}(b)&=\frac{9}{2}f\rq{}\rq{}(b)\lambda(b).
\end{aligned}
\end{equation}
By substituting $\lambda=\frac{g\rq{}}{f}$, the last two equations in \eqref{functions} give a system of ODEs, that we solve uniquely with some initial conditions  
$$f(0)=s,\,g\rq{}(0)=t,\,g\rq{}\rq{}(0)=r,\,s,t,\neq 0,$$ 
and then find $\lambda$ uniquely. 

All the normalization conditions \eqref{normalform} are satisfied now. We are now in the position to prove the main result of this section.

\begin{theorem}\label{convergent}
For any smooth manifold of solutions $\mathcal M$ defined near a point $p$ there exists a smooth product transformation \eqref{coupled} 
mapping $\mathcal M$ onto a smooth manifold of solutions in the normal form \eqref{normalform}. A normalizing transformation is uniquely determined by the restriction of its differential onto the leaf $T^C_p \mathcal M$ of the canonical subbundle $T^C\mathcal M\subset T\mathcal M$, and by its transverse second order derivative. In turn, in any coordinates \eqref{approximation}, this amounts to a choice of the five parameters \eqref{mapdata}.

If $\mathcal M$ is real-analytic, then the normalizing transformation and the normal form are real-analytic as well.
\end{theorem}

\begin{proof}
The existence statement is already proved above, so that it remains to prove the uniqueness in the general smooth case (in the real-analytic case the uniqueness statements  immediately follows from the formal \autoref{formal}).  It is sufficient to prove that if two smooth manifolds of solutions are in the normal form and $H=(f,g,\lambda,\mu)$ is a map between them with the identity Taylor expansion at $0$, then $H$ is the identity. 

 Note that since the formal expansion of $H$ is identical, it preserves the direction given by \eqref{Gamma},  hence it preserves the curve \eqref{Gamma} itself (since there is a unique chain in a given direction, and the family of chains in $\mathcal M$ is invariant under smooth transformations). Thus we have $f=x(1+O(1)),\,\lambda=a(1+O(1))$. Consider then the basic identity \eqref{identity}. 
 Putting in \eqref{identity} $a=0$ and then $x=0$, we get, respectively,
 $$g(x,b)=\mu(0,b),\quad g(0,b)=\mu(a,b).$$
 Thus we have $g=g(y),\,\mu=g(b)$. Differentiating then \eqref{identity} once in $a$ and plugging $a=0$, and then doing so for $x$, we similarly obtain $f=xf(y),\,\lambda=a\lambda(b)$. Proceeding after that similarly to Step 5 above, we obtain the system \eqref{functions}, which we need to solve with $f(0)=1=g\rq{}(0)=1,\,g\rq{}\rq{}(0)=0,$ so that 
 we finally get (by the uniqueness for solutions) $f(b)=\lambda(b)=1,\,g(b)=b$, as required. 
\end{proof}

By inspecting the proof of \autoref{convergent} and Step 5 above, one can see that

\bigskip

\noindent\emph{we can canonically, up to the action of the projective group}

\begin{equation}\label{parameter}
x\mapsto \frac{sx}{1+ry},\quad y\mapsto \frac{sty}{1+ry},\quad a\mapsto \frac{ta}{1+rb},\quad b\mapsto \frac{stb}{1+rb}
\end{equation}

\bigskip

\noindent\emph{choose a parameterization on each chain $\tilde\gamma$.}

\bigskip

Note that the group \eqref{parameter} is obtained from the one generated by \eqref{g0},\eqref{g+} by putting $\alpha=\beta=0$ (which corresponds to fixing a chain through the reference point).

We are now in the position to prove \autoref{theor1}.

\begin{proof}[Proof of \autoref{theor1}]
We note that if $\mathcal M$, as in \eqref{manifold}, is a manifold of solutions for an ODE $\mathcal E$, as in \eqref{ODE}, then $\mathcal E$ may be obtained from $\mathcal M$ by solving the system
\begin{equation}\label{findab}
y=\Phi(x,a,b),\quad y\rq{}=\Phi_x(x,a,b)
\end{equation}
for $a,b$, and substituting the resulting functions $a=A(x,y,y\rq{}),\,b=B(x,y,y\rq{})$ into $y\rq{}\rq{}=\Phi_{xx}(x,a,b)$. Thus we have
\begin{equation}\label{MtoE}
F(x,y,y\rq{})=\Phi_{xx}\bigl(x,A(x,y,y\rq{}),B(x,y,y\rq{})\bigr).
\end{equation}

Let us now consider the given ODE $\mathcal E$ defined near the origin in $ J^1(\RR{},\RR{})$, and a manifold $\mathcal M$ of solutions of it obtained by some admissible parameterization. If  $H$ is the smooth transformation (existing by \autoref{convergent}) with some data \eqref{mapdata} bringing  $\mathcal M$ into a normal form $\mathcal M_N$, as in \eqref{normalform}, then from \eqref{findab} we find
$$a=y\rq{}+O\left(x^3y\rq{}^2\right)+O\left(xy\rq{}^4\right),\quad b=y+xy\rq{}+O\left(x^3y\rq{}^2\right)+O\left(xy\rq{}^4\right).$$
Substituting this into \eqref{MtoE} gives (in view of $\Phi_{xx}=O(x^2a^4)+O(x^4a^2)$ for a manifold of solutions in the normal form) a function $F(x,y,y\rq{})$ as in \eqref{normalODE}. It means that the transformation given by the first  two components of $H$ brings $\mathcal E$ to the normal form, and this proves the existence claim of the theorem. 

For the uniqueness we may consider a factorization similar to \eqref{factor} for all possible maps into a normal form, and then it is sufficient yo prove the uniqueness statement for maps $G=(f,g)$ satisfying \eqref{normalmap}. Let such a map $g$ maps the initial  ODE $\mathcal E$ into a normal form $\mathcal E_N$. We then choose for the normal form $\mathcal E_N$ the parameterization of solutions given by
$$(a,b)=\bigl(y\rq{}(0),y(0)\bigr).$$
Thus for the corresponding manifold of solutions $\mathcal M_N=\{y=\Phi^*(x,a,b)\}$ we have
\begin{equation}\label{goodparam}
\Phi^*_{0l}(b)=\Phi^*_{1l}(b)=0,\quad l\geq 0.
\end{equation}
  We now substitute $y=\Phi^*(x,a,b)$ into the ODE $\mathcal E_N$   and collect all terms involved in the normalization conditions \eqref{normalform}. From here, by using \eqref{normalODE} and \eqref{goodparam}, it is not 
   difficult to verify that all the conditions \eqref{normalform} are satisfied. Hence $\mathcal M_N$ is in the normal form \eqref{normalform}, and the uniqueness follows from \autoref{convergent}. This completes the proof.
\end{proof}
\begin{proof}[Proof of Theorem 7]
The proof is obtained by combining \autoref{uptoZ2} and the convergence  \autoref{convergent}, and is very analogues to the proof of \autoref{theor1}. We leave the details to the reader.
\end{proof}

\end{document}